\documentclass[english]{amsart}
\usepackage[T1]{fontenc}
\usepackage[latin9]{inputenc}
\usepackage{mathrsfs}
\usepackage{amsthm}
\usepackage{amsbsy}
\usepackage{amssymb}
\usepackage{setspace}
\usepackage[all]{xy}
\makeatletter
\numberwithin{equation}{section}
\numberwithin{figure}{section}
\usepackage{enumitem}		
\theoremstyle{plain}
\newtheorem{thm}{\protect\theoremname}
  \theoremstyle{plain}
  \newtheorem*{thm*}{\protect\theoremname}
  \theoremstyle{plain}
  \newtheorem{lem}[thm]{\protect\lemmaname}
  \theoremstyle{remark}
  \newtheorem{rem}[thm]{\protect\remarkname}
  \theoremstyle{plain}
  \newtheorem{prop}[thm]{\protect\propositionname}
  \theoremstyle{plain}
  \newtheorem{cor}[thm]{\protect\corollaryname}

\makeatother

\usepackage{babel}
  \providecommand{\corollaryname}{Corollary}
  \providecommand{\lemmaname}{Lemma}
  \providecommand{\propositionname}{Proposition}
  \providecommand{\remarkname}{Remark}
  \providecommand{\theoremname}{Theorem}
\providecommand{\theoremname}{Theorem}

\begin{document}

\title{Langlands parameters associated to special maximal parahoric spherical
representations}

\author{Manish Mishra}

\email{mmishra@math.huji.ac.il}

\address{Einstein Institute of Mathematics, The Hebrew University of Jerusalem,
Jerusalem, 91904, Israel}
\begin{abstract}
We describe the image, under the local Langlands correspondence for
tori, of the characters of a torus which are trivial on its Iwahori
subgroup. Let $k$ be a non-archimedian local field. Let $\boldsymbol{G}$
be a connected reductive group defined over $k$, which is quasi-split
and split over a tamely ramified extension. Let $K$ be a special
maximal parahoric subgroup of $\boldsymbol{G}(k)$. To the class of
representations of $\boldsymbol{G}(k)$, having a non-zero vector
fixed under $K$, we establish a bijection, in a natural way, with
the twisted semi-simple conjugacy classes of the inertia fixed subgroup
of the dual group $\hat{\boldsymbol{G}}$. These results generalize
the well known classical results to the ramified case. 
\end{abstract}
\maketitle

\section*{Introduction}

Let $k$ be a non-archimedian local field. Let $\boldsymbol{G}$ be
an unramified connected reductive group defined over $k$ and let
$\pi$ be a smooth, irreducible, admissible representation of $\boldsymbol{G}$
which is unramified, i.e., there is a hyperspecial subgroup $K$ of
$\boldsymbol{G}(k)$ such that the $K$-invariant subspace $\pi^{K}$
of the space realizing $\pi$ is non-zero. One can associate a Langlands
parameter to the representation $\pi$ via the following recipe.

The representation $\pi$ corresponds, in a natural way, to a character
of the Hecke algebra $\mathscr{H}(\boldsymbol{G}(k),K)$. Then via
the Satake isomorphism, this character corresponds to an unramified
character $\chi$ of $\boldsymbol{T}(k)$ for certain maximal torus
$\boldsymbol{T}$. $ $This character $\chi$ is unique up to the
relative Weyl group conjugacy. The character $\chi$, under the local
Langlands correspondence for tori, corresponds to a Langlands parameter\linebreak{}
$\varphi_{\chi}\in\mbox{H}^{1}(W_{k},\hat{\boldsymbol{T}})$, where
$\hat{\boldsymbol{T}}$ is the dual torus and $W_{k}$ is the Weil
group. This parameter is induced from a cocycle in $\mbox{H}^{1}(W_{k}/I_{k},\hat{\boldsymbol{T}})$,
where $I_{k}$ is the inertia subgroup of $W_{k}$. Using this, one
can then associate to the parameter $\varphi_{\chi}$ a semisimple
conjugacy class in $\hat{\boldsymbol{G}}\rtimes(W_{k}/I_{k})$, where
$\hat{\boldsymbol{G}}$ is the complex dual of $\boldsymbol{G}$.
This semisimple conjugacy class describes the Langlands parameter
associated to $\pi$. 

All these classic results are well known and can be found in \cite{MR546593}.
We wish to find analogous statements when $\boldsymbol{G}$ is not
necessarily unramified. Let $\boldsymbol{G}$ be quasi-split and tamely
ramified. Let $K$ be a special maximal parahoric subgroup of $\boldsymbol{G}$
and let $\pi$ be a smooth, irreducible, admissible representation
of $\boldsymbol{G}$ which is $K$-spherical, i.e., $\pi^{K}\neq0$.
We associate to $\pi$, a character $\chi$ of $\boldsymbol{T}(k)$
of a certain maximal torus $\boldsymbol{T}$ in a similar way as above,
using the description of special maximal parahoric Hecke algebras
given in \cite{haines2010satake}. We show that $\chi$ is trivial
on the Iwahori subgroup $\boldsymbol{T}(k)_{0}$ of $\boldsymbol{T}(k)$.
In Theorem-1, we calculate the image, under the local Langlands correspondence
for tori, of all such characters which are trivial on $\boldsymbol{T}(k)_{0}$
and show that it is precisely inflation of the cocycles in $H^{1}(W_{k}/I_{k},\hat{\boldsymbol{T}}^{I_{k}})$.
In Theorem-2, We show that the orbits in $H^{1}(W_{k}/I_{k},\hat{\boldsymbol{T}}^{I_{k}})$
of the relative Weyl group are in bijection with the semisimple conjugacy
classes in $\hat{\boldsymbol{G}}^{I}\rtimes(W_{k}/I_{k})$.

In \cite[Chapter 11]{MR2498783}, a character of $\boldsymbol{T}(k)$
is called elementary if under the Langlands reciprocity map, it corresponds
to a cocycle in $H^{1}(W_{k},\hat{\boldsymbol{T}})$ which is the
inflation of a cocyle in $H^{1}(W_{k}/I_{k},\hat{\boldsymbol{T}}^{I_{k}})$.
The question of characterizing the elementary characters is a natural
one and some partial results are presented in \cite[chapter 11]{MR2498783}.
Our first theorem answers this question by showing that a character
is elementary if and only if it is trivial on the Iwahori subgroup
$\boldsymbol{T}(k)_{0}$. 

The referee has pointed out to me that Lemma \ref{lem5} for arbitrary
tori was independently observed in a recent work by Haines \cite[Sec. 3.3.1]{haines}.
In that article, Haines mentions that the result is implied by a more
general result of Kaletha \cite[Prop. 4.5.2]{kaletha}. A characteristic
zero assumption is made in that result of Kaletha.

\section{Notations}

Let $k$ be a non-archimedian local field. Let $\boldsymbol{G}$ be
a connected reductive group defined $k$, which is quasi-split and
split over a tamely ramified extension. We denote $\boldsymbol{G}(k)$
by $G$ and likewise for other algebraic groups. Let $K$ be a special
maximal parahoric subgroup of $G$ corresponding to a special vertex
$\nu$ in the Bruhat-Tits building $\mathcal{B}(G_{ad})$. Let $\boldsymbol{A}$
denote a maximal split $k$-torus whose corresponding apartment in
$\mathcal{B}(G_{ad})$ contains $\nu$. Let $\boldsymbol{T}=Z_{\boldsymbol{G}}\boldsymbol{A}$,
the centralizer of $\boldsymbol{A}$ in $\boldsymbol{T}$. Then $\boldsymbol{T}$
is a maximal torus in $\boldsymbol{G}$ since $\boldsymbol{G}$ is
quasi-split. Let $W$ denote the relative Weyl group of $\boldsymbol{G}$.
Let $\mathscr{H}(G,K)$ be the Hecke algebra of $K$-bi-invariant
compactly supported complex-valued functions on $G$. Let $T_{c}$
and $T_{0}$ denote respectively the maximal compact subgroup and
the Iwahori subgroup of $T$. Let $\hat{\boldsymbol{G}}$ denote the
complex dual of $\boldsymbol{G}$ and $\hat{\boldsymbol{G}}_{ss}$
the set of semisimple elements in $\hat{\boldsymbol{G}}$. Let Inn$(\hat{\boldsymbol{G}})$
be the group of inner automorphisms of $\hat{\boldsymbol{G}}$. Let
$\sigma=\sigma_{k}$ denotes the Frobenius element in $W_{k}/I_{k},$
where $W_{k}$ is the Weil group of $k$ and $I=I_{k}$ is its inertia
subgroup. We denote the identity component of an algebraic group $\mathscr{G}$
by $\mathscr{G}^{\circ}$.

\section{Statement of the theorems}
\begin{thm}
A character is elementary if and only if it is trivial on the Iwahori
subgroup. In other words, we have a commutative diagram:

\[
\xymatrix{\mbox{Hom}(T,\mathbb{C}^{\times})\ar@{->}[r]_{\cong}^{LLC} & H^{1}(W_{k},\hat{\boldsymbol{T}})\\
\mbox{Hom}(T/T_{0},\mathbb{C^{\times}})\ar@{^{(}->}[u]\ar@{->}[r]^{\cong} & H^{1}(W_{k}/I,\hat{\boldsymbol{T}}^{I}),\ar@{^{(}->}^{infl}[u]
}
\]
 where $LLC$ is the local Langlands correspondence for tori and $infl$
is the inflation homomorphism. 
\end{thm}
Let Rep$(G)$ denote the set of equivalent classes of smooth, irreducible
admissible representations of $G.$ 
\begin{thm}
\label{main th}The $K$-spherical representations are in a natural
bijection with the semisimple conjugacy classes in $\hat{\boldsymbol{G}}^{I}\rtimes\sigma$
via the local Langlands correspondence for tori. More precisely, 
\begin{eqnarray*}
\{\pi|\pi\in\mbox{Rep}(G),\mbox{ \ensuremath{\pi^{K}\neq0\}}} & \leftrightarrow & \mbox{Hom}_{\mathbb{C}}(\mathscr{H}(G,K),\mathbb{C})\\
 & \cong & \mbox{Hom}(T/T_{0},\mathbb{C}^{\times})/W\\
 & \cong & (\hat{\boldsymbol{T}}^{I})_{\sigma}/W\\
 & \cong & (\hat{\boldsymbol{G}}^{I}\rtimes\sigma)_{ss}/\mbox{Inn}(\hat{\boldsymbol{G}}^{I}).
\end{eqnarray*}
 
\end{thm}

\section{Langlands correspondence for tori}

The following treatment of Local Langlands correspondence for tori
can be found in \cite{MR2508725}.

\subsection{The special case of an induced torus. \label{induced torus}}

Let $\boldsymbol{T}=R_{k^{\prime}/k}\mathbb{G}_{m}$ be an induced
torus, when $k^{\prime}$ is a finite separable extension of $k$.
Then $T=k^{\prime\times}$ and the group of characters $X^{*}(\boldsymbol{T})$
is canonically a free $\mathbb{Z}$-module with basis $W_{k}/W_{k^{\prime}}$.
From this, it follows that $\hat{\boldsymbol{T}}$ is canonically
isomorphic to $ind_{W_{k^{\prime}}}^{W_{k}}\mathbb{C}^{\times}$.
We get, 
\begin{eqnarray}
\mbox{Hom}(T,\mathbb{C^{\times}}) & \cong & \mbox{Hom}(k^{\prime\times},\mathbb{C}^{\times})\nonumber \\
 & \cong & \mbox{Hom}(W_{k^{\prime}},\mathbb{C}^{\times})\label{class field}\\
 & \cong & \mbox{H}^{1}(W_{k^{\prime}},\mathbb{C}^{\times})\nonumber \\
 & \cong & \mbox{H}^{1}(W_{k},ind_{W_{k^{\prime}}}^{W_{k}}\mathbb{C}^{\times})\label{shapiro}\\
 & \cong & \mbox{H}^{1}(W_{k},\hat{\boldsymbol{T})}.\nonumber 
\end{eqnarray}
 The isomorphism \ref{class field} follows by class field theory
and the isomorphism \ref{shapiro} by Shapiro's lemma.

\subsection{The LLC for tori in general}
\begin{thm*}
\cite[7.5 Theorem]{MR2508725}There is a unique family of homomorphisms
\[
\varphi_{\boldsymbol{T}}:Hom(T,\mathbb{C}^{\times})\rightarrow H^{1}(W_{k},\hat{\boldsymbol{T}})
\]
with the following properties:
\begin{enumerate}
\item $\varphi_{\boldsymbol{T}}$ is additive functorial in $\boldsymbol{T}$,
i.e., it is a morphism between two additive functors from the category
of tori over $k$ to the category of abelian groups;
\item For $\boldsymbol{T}=R_{k^{\prime}/k}\mathbb{G}_{m}$, where $k^{\prime}/k$
is a finite separable extention, $\varphi_{\boldsymbol{T}}$ is the
isomorphism described in Section \ref{induced torus}. 
\end{enumerate}
\end{thm*}

\section{Proof of theorem 1}
\begin{lem}
\label{main lemma}Let $\boldsymbol{T}$ be a tous defined over $k$.
Then there exists an isomorphism
\[
\kappa_{\boldsymbol{T}}:\mbox{Hom}(T/T_{0},\mathbb{C}^{\times})\rightarrow\mbox{H}^{1}(W_{k}/I,\hat{\boldsymbol{T}}^{I}).
\]
 Moreover, the isomorphism $k_{\boldsymbol{T}}$ is additive functorial
in $\boldsymbol{T}.$ \end{lem}
\begin{proof}
We have by Kottwitz isomorphism \cite[sec 7]{MR1485921}, (see also
\cite[Prop 1.0.2]{haines2010satake}) 
\begin{eqnarray}
T/T_{0} & \cong & ((X^{*}(\hat{\boldsymbol{T}}))_{I})^{\sigma}\nonumber \\
 & \cong & X^{*}((\hat{\boldsymbol{T}}^{I})_{\sigma}).\label{kottwitz}
\end{eqnarray}

Therefore,
\begin{eqnarray}
\mbox{Hom}(T/T_{0},\mathbb{C}^{\times}) & \cong & \mbox{Hom}(X^{*}((\hat{\boldsymbol{T}}^{I})_{\sigma}),\mathbb{C^{\times}})\nonumber \\
 & \cong & (\hat{\boldsymbol{T}}^{I})_{\sigma}\label{cartier duality}\\
 & \cong & \mbox{H}^{1}(W_{k}/I,\hat{\boldsymbol{T}}^{I}).\nonumber 
\end{eqnarray}
\label{seq of isom}

The isomorphism in Equation \ref{cartier duality} is by Cartier duality.
The functoriality of $\kappa_{\boldsymbol{T}}$ follows from the functoriality
of the Kottwitz isomorphism. \end{proof}
\begin{rem}
From the relation 
\[
T/T_{c}\cong(X^{*}(\hat{\boldsymbol{T}})_{I})^{\sigma}/\mbox{torsion},
\]
 one similarly obtains the isomorphism
\begin{equation}
\mbox{Hom}(T/T_{c},\mathbb{C}^{\times})\cong\mbox{H}^{1}(W_{k}/I,(\hat{\boldsymbol{T}}^{I})^{\circ}).
\end{equation}
\end{rem}
\begin{lem}
\label{lem5}Let $k^{\prime}/k$ be a finite separable extension and
put $\boldsymbol{T}=R_{k^{\prime}/k}(\mathbb{G}_{m})$. Then the isomorphism
$\mbox{Hom}(T/T_{0},\mathbb{C}^{\times})\cong\mbox{H}^{1}(W_{k}/I,\hat{\boldsymbol{T}}^{I})$
obtained from the Kottwitz isomorphism in Lemma \ref{main lemma}
is the same as the one induced from the Local Langlands correspondence
for tori. \end{lem}
\begin{proof}
Since $\boldsymbol{T}$ is an induced torus, $X^{*}(\boldsymbol{T})$
is canonically a free $\mathbb{Z}$-module with basis $W_{k}/W_{k^{\prime}}$.
Consequently, $\hat{\boldsymbol{T}}$ is simply $ind_{W_{k^{\prime}}}^{W_{k}}\mathbb{C}^{\times}$.
Let $\mathscr{O}$ and $\mathscr{O}^{\prime}$ be the ring of integers
in $k$ and $k^{\prime}$ respectively. We have, 
\begin{eqnarray}
H^{1}(W_{k}/I_{k},\hat{\boldsymbol{T}}^{I_{k}}) & \cong & (\hat{\boldsymbol{T}}^{I_{k}})_{\sigma_{k}}\nonumber \\
 & \cong & ((ind_{W_{k^{\prime}}}^{W_{k}}\mathbb{C}^{\times})^{I_{k}})_{\sigma_{k}}\nonumber \\
 & \cong & \mbox{H}^{1}(W_{k}/I_{k},(ind_{W_{k^{\prime}}}^{W_{k}}\mathbb{C}^{\times})^{I_{k}})\nonumber \\
 & \cong & \mbox{H}^{1}(W_{k^{\prime}}/I_{k^{\prime}},\mathbb{C}^{\times})\label{shapiro-1}\\
 & \cong & \mbox{Hom}(k^{\prime\times}/\mathscr{O}^{\prime\times},\mathbb{C}^{\times})\label{class-field}\\
 & \cong & \mbox{Hom}(T/T_{0},\mathbb{C}^{\times}).\nonumber 
\end{eqnarray}
Here, the isomorphism in \ref{shapiro-1} is induced by the isomorphism
in Shapiro's lemma:
\[
\xymatrix{\mbox{H}^{1}(W_{k},\mbox{ind}{}_{W_{k^{\prime}}}^{W_{k}}\mathbb{C}^{\times})\ar@{->}[r]^{\sim} & \mbox{H}^{1}(W_{k^{\prime}},\mathbb{C}^{\times})\\
\mbox{H}^{1}(W_{k}/I_{k},(\mbox{ind}{}_{W_{k^{\prime}}}^{W_{k}}\mathbb{C}^{\times})^{I_{k}})\ar@{->}[r]^{\sim}\ar@{^{(}->}[u]^{infl} & \mbox{H}^{1}(W_{k^{\prime}}/I_{k^{\prime}},\mathbb{C}^{\times})\ar@{^{(}->}[u]^{infl}.
}
\]
 Thus the Local Langlands Correspondence also induces an isomorphism: 

\[
\varphi_{\boldsymbol{T}}:\mbox{Hom}(T/T_{0},\mathbb{C}^{\times})\cong\mbox{H}^{1}(W_{k}/I,\hat{\boldsymbol{T}}^{I}).
\]

The Kottwitz isomorphism for induced tori is constructed as follows
(see \cite[Sec. 7.2]{MR1485921}): the homomorphism 
\[
v_{T}^{\prime}:T\rightarrow\mbox{Hom}(X^{*}(\boldsymbol{T}),\mathbb{Z}),
\]
 sending $t\in T$ to the homomorphism 
\[
\lambda\mapsto\mbox{ord}(\lambda(t)),
\]
 induces an isomorphism 
\begin{equation}
v_{T}:T/T_{0}\cong\mbox{Hom}(X^{*}(\boldsymbol{T})^{I},\mathbb{Z})^{\sigma}.\label{kott hom-1}
\end{equation}
 Also, the homomorphism 
\[
q_{T}^{\prime}:X_{*}(\boldsymbol{T})\rightarrow\mbox{Hom}(X^{*}(\boldsymbol{T}),\mathbb{Z}),
\]
 which sends $\mu\in X_{*}(T)$ to the homomorphism 
\[
\lambda\mapsto\langle\lambda,\mu\rangle,
\]
 induces an isomorphism 
\begin{equation}
q_{T}:(X_{*}(\boldsymbol{T})_{I})^{\sigma}\cong\mbox{Hom}(X^{*}(\boldsymbol{T})^{I},\mathbb{Z})^{\sigma}.\label{kott hom-2}
\end{equation}
 From the Equations \ref{kott hom-1} and \ref{kott hom-2}, we get
the Kottwitz isomorphism for the induced torus $\boldsymbol{T}$:

\[
w_{T}:T/T_{0}\cong X^{*}((\hat{\boldsymbol{T}}^{I})_{\sigma}),
\]
 where we used the identifications $(X_{*}(\boldsymbol{T})_{I})^{\sigma}=(X^{*}(\hat{\boldsymbol{T}})_{I})^{\sigma}\cong X^{*}((\hat{\boldsymbol{T}}^{I})_{\sigma})$.
This map $w_{T}$ induces the map $\kappa_{\boldsymbol{T}}:\mbox{Hom}(T/T_{0},\mathbb{C}^{\times})\cong\mbox{H}^{1}(W_{k}/I,\hat{\boldsymbol{T}}^{I})$
in Lemma \ref{main lemma}. 

We identify $T/T_{0}$ with $\mathbb{Z}$ by the isomorphisms 
\begin{equation}
T/T_{0}\cong k^{\prime\times}/\mathscr{O}^{\prime\times}\cong\varpi^{\prime\mathbb{Z}}\cong\mathbb{Z}.\label{Z isom-1}
\end{equation}
 Here $\varpi^{\prime}$ is the uniformizer in $k^{\prime}.$ Let
$J$ be a $W_{k}$-stable basis of $X^{*}(\boldsymbol{T}).$ Choose
any $\lambda\in J$ and let $J_{\lambda}\subset J$ be the orbit or
$\lambda$ in $J$ under the action of $I_{k}.$ Let $\chi=\underset{\mu\in J_{\lambda}}{\sum\mu}$.
Then $\chi\in X^{*}(\hat{\boldsymbol{T}})^{I_{k}}$. 

Let $f\in\mbox{Hom}(T/T_{0},\mathbb{C}^{\times})$ and let $c=f(1)$
(under the identification in equation \ref{Z isom-1}). Then both
$\kappa_{\boldsymbol{T}}$ and $\varphi_{\boldsymbol{T}}$ map $f$
to the cocycle $\phi_{f}\in\mbox{H}^{1}(W_{k}/I,\hat{\boldsymbol{T}}^{I})$
defined by $\sigma\mapsto\chi\otimes c.$ Thus $\kappa_{\boldsymbol{T}}=\varphi_{\boldsymbol{T}}.$
This completes the proof of the lemma. \end{proof}
\begin{prop}
There is a unique family of homomorphisms 
\[
\varphi_{\boldsymbol{T}}:Hom(T/T_{0},\mathbb{C}^{\times})\rightarrow H^{1}(W_{k},\hat{\boldsymbol{T}}),
\]
with the following properties:
\begin{enumerate}
\item $\varphi_{\boldsymbol{T}}$ is additive functorial in $\boldsymbol{T}$,
i.e., it is a morphism between two additive functors from the category
of tori over $k$ to the category of abelian groups.
\item For $\boldsymbol{T}=R_{k^{\prime}/k}\mathbb{G}_{m}$, where $k^{\prime}/k$
is a finite Galois extension, $\varphi_{\boldsymbol{T}}$ is the homomorphism
induced from the Local Langlands correspondence for Tori. 
\end{enumerate}
\end{prop}
\begin{proof}
Since the isomorphism $\kappa_{\boldsymbol{T}}$ in Lemma \ref{main lemma}
is additive funtorial in $\boldsymbol{T}$, we thus get an additive
functorial family of homomorphisms
\[
\xymatrix{\varphi_{\boldsymbol{T}}:\mbox{Hom}(T/T_{0},\mathbb{C}^{\times})\ar@{->}[r] & \mbox{H}^{1}(W_{k}/I,\hat{\boldsymbol{T}}^{I})\ar@{^{(}->}^{\mbox{infl}}[r] & \mbox{H}^{1}(W_{k},\hat{\boldsymbol{T}}).}
\]

This shows existence. To show uniqueness, let $\boldsymbol{T}$ be
a given torus defined over $k$ and let $k^{\prime}/k$ be a finite
Galois extension such that $\boldsymbol{T}$ is split over $k^{\prime}$.
Let\\
 $\boldsymbol{T}^{\prime}=R_{k^{\prime}/k}(\boldsymbol{T}\otimes_{k}k^{\prime})$.
Then $\boldsymbol{T}^{\prime}$ is isomorphic to a direct sum of $d=$
dim($\boldsymbol{T}$) tori of the form $R_{k^{\prime}/k}(\mathbb{G}_{m})$
and there is a natural embedding $\boldsymbol{T}\hookrightarrow\boldsymbol{T}^{\prime}$.
This gives an embedding $T/T_{0}\hookrightarrow T^{\prime}/T_{0}^{\prime}$.
By (1), there is a commutative diagram
\[
\xymatrix{\mbox{Hom}(T/T_{0},\mathbb{C}^{\times})\ar@{->}[r]^{\varphi_{\boldsymbol{T}}} & \mbox{H}^{1}(W_{k},\hat{\boldsymbol{T}})\\
\mbox{Hom}(T^{\prime}/T_{0}^{\prime},\mathbb{C}^{\times})\ar@{->>}[u]\ar@{->}[r]^{\varphi_{\boldsymbol{T}^{\prime}}} & \mbox{H}^{1}(W_{k},\hat{\boldsymbol{T}^{\prime}}).\ar@{->}[u]
}
\]
 Notice that $\varphi_{\boldsymbol{T}^{\prime}}$ is completely determined
by (2). Now given $x\in\mbox{Hom}(T/T_{0},\mathbb{C}^{\times})$,
we can lift it to $x^{\prime}\in\mbox{Hom}(T^{\prime}/T_{0}^{\prime},\mathbb{C}^{\times})$.
It follows that $\varphi_{\boldsymbol{T}}(x)$ is the image of $\varphi_{\boldsymbol{T}^{\prime}}(x^{\prime})$
under the vertical arrow on the right, and is hence determined by
(1) and (2). \end{proof}
\begin{thm}
Let $\boldsymbol{T}$ be a torus defined over $k$. Then the local
Langlands correspondence for tori induces the isomorphism
\[
\mbox{Hom}(T/T_{0},\mathbb{C}^{\times})\cong\mbox{H}^{1}(W_{k}/I,\hat{\boldsymbol{T}}^{I}).
\]
 \end{thm}
\begin{proof}
By the above Proposition, it follows that the homomorphism\linebreak{}
 $\mbox{Hom}(T/T_{0},\mathbb{C}^{\times})\rightarrow\mbox{H}^{1}(W_{k},\hat{\boldsymbol{T}})$,
determined by the Kottwitz map must be the same as the one determined
by LLC. Therefore, the image of $\mbox{Hom}(T/T_{0},\mathbb{C}^{\times})$
under LLC must be the same as the one determined by $\kappa_{\boldsymbol{T}}$
and which is $\mbox{H}^{1}(W_{k}/I,\hat{\boldsymbol{T}}^{I})$. This
proves the theorem.
\end{proof}

\section{\label{new_sec}Some results about inertia fixed groups}

Let $\boldsymbol{B}$ be a Borel subgroup of $\boldsymbol{G}$ containing
the maximal torus $\boldsymbol{T}$ of $\boldsymbol{G}$. The triple
$(\boldsymbol{G},\boldsymbol{B},\boldsymbol{T})$ determines a based
root datum. Let $(\hat{\boldsymbol{G}},\hat{\boldsymbol{B}},\hat{\boldsymbol{T}})$
be the triple defined over $\mathbb{C}$ which corresponds to the
dual based root datum. Since $\boldsymbol{G}$ is tamely ramified,
the inertia group acts on $(\hat{\boldsymbol{G}},\hat{\boldsymbol{B}},\hat{\boldsymbol{T}})$
via a cyclic group $(\tau)$. Let $\boldsymbol{W}$ be the Weyl group
of $\hat{\boldsymbol{G}}$. Let $\hat{\boldsymbol{T}}_{ad}$ be the
image of $\hat{\boldsymbol{T}}$ in $\hat{\boldsymbol{G}}_{ad}$,
the adjoint group of $\hat{\boldsymbol{G}}$. Let $\hat{\boldsymbol{Z}}$
be the center of $\hat{\boldsymbol{G}}$. 

The following results are given in \cite[Sec. 1.1 (Theorem 1.1.A and the paragraphs following it)]{Ko-Sh}. 
\begin{thm}
[Kottwitz-Shelstad]\label{referee results}\item[$(1)$]$(\hat{\boldsymbol{G}}^{\tau})^{\circ}$
is a connected reductive group.

\item[$(2)$]$(\hat{\boldsymbol{B}}^{\tau})^{\circ}$ is a Borel subgroup
of $(\hat{\boldsymbol{G}}^{\tau})^{\circ}$, containing its maximal
torus $(\hat{\boldsymbol{T}}^{\tau})^{\circ}$.

\item[$(3)$]$\boldsymbol{W}^{\tau}$ is the Weyl group of $(\hat{\boldsymbol{G}}^{\tau})^{\circ}$.

\item[$(4)$]$\hat{\boldsymbol{G}}^{\tau}=\hat{\boldsymbol{Z}}^{\tau}(\hat{\boldsymbol{G}}^{\tau})^{\circ}$.

\item[$(5)$]$(\hat{\boldsymbol{T}}_{ad})^{\tau}$ is connected.

From $(5)$, we immediately obtain

\item[$(5')$]$\hat{\boldsymbol{T}}^{\tau}=\hat{\boldsymbol{Z}}^{\tau}(\hat{\boldsymbol{T}}^{\tau})^{\circ}$.\end{thm}
\begin{rem}
As pointed out by the referee, these facts can be proved as follows.
The fact that $(\hat{\boldsymbol{T}}_{ad})^{\tau}$ is connected follows
from \cite[remark at the end of Cor. 9.12]{Steinberg}. Using this
and \cite[Lemma 4.6]{Zhu}, it follows that $(\hat{\boldsymbol{G}}_{ad})^{\tau}$
is connected. From this, $(4)$ follows. The fact that $(\hat{\boldsymbol{G}}^{\tau})^{\circ}$
is reductive follows from a general lemma of Kottwitz \cite[10.1.2]{MR757954}.
The fact that $\boldsymbol{W}^{\tau}$ is the Weyl group of $(\hat{\boldsymbol{G}}^{\tau})^{\circ}$
can be proved by adapting the proof of Theorem 8.2 of \cite{Steinberg}.
Since $\tau$ fixes a splitting, this allows one to take $t=1$ in
that proof. 
\end{rem}

\section{Proof of theorem 2}
\begin{lem}
$\mbox{Hom}_{\mathbb{C}}(\mathscr{H}(G,K),\mathbb{C})\cong\mbox{Hom}(T/T_{0},\mathbb{C}^{\times})/W.$\end{lem}
\begin{proof}
$\mathscr{H}(G,K)\cong\mathbb{C}[T/T_{0}]^{W}$ by \cite[Theorem 1.0.1]{haines2010satake}.
Therefore 
\begin{eqnarray*}
\mbox{Hom}_{\mathbb{C}}(\mathscr{H}(G,K),\mathbb{C}) & \cong & \mbox{Hom}_{\mathbb{C}}(\mathbb{C}[T/T_{0}]^{W},\mathbb{C})\\
 & \cong & \mbox{Hom}_{\mathbb{C}}(\mathbb{C}[T/T_{0}],\mathbb{C})/W\\
 & \cong & \mbox{Hom}(T/T_{0},\mathbb{C}^{\times})/W.
\end{eqnarray*}
 \end{proof}
\begin{prop}
\label{main_prop}
\[
(\hat{\boldsymbol{T}}^{I})_{\sigma}/W\cong Z(\hat{\boldsymbol{G}})^{I}(\hat{\boldsymbol{G}}^{I})_{ss}^{\circ}\rtimes\sigma/\mbox{Inn}(Z(\hat{\boldsymbol{G}})^{I}(\hat{\boldsymbol{G}}^{I})^{\circ}).
\]

\begin{proof}
Let $\boldsymbol{W}$ be, as before, the Weyl group of $\hat{\boldsymbol{G}}$.
By Theorem \ref{referee results}$(3)$, $\boldsymbol{W}^{\tau}$
is the Weyl group of $\mathrm{(\hat{\boldsymbol{G}}^{\tau})^{\circ}}$.
Since $(\hat{\boldsymbol{G}}^{\tau})^{\circ}$ is reductive (by Theorem
\ref{referee results}$(1)$), it follows from the proof of \cite[Lemma 6.5 ]{MR546608}
that we have a surjection
\begin{equation}
(\mathrm{\hat{\boldsymbol{T}}}^{\tau})^{\circ}\twoheadrightarrow(\mathrm{(\hat{\boldsymbol{G}}^{\tau})^{\circ}\rtimes\sigma)_{ss}/\mbox{Inn}(\mathrm{(\hat{\boldsymbol{G}}^{\tau})^{\circ})}}.
\end{equation}
 By Theorem \ref{referee results}$(5')$, this implies, 
\begin{equation}
\mathrm{\hat{\boldsymbol{T}}}^{\tau}\twoheadrightarrow(\mathrm{\hat{\boldsymbol{Z}}^{\tau}(\hat{\boldsymbol{G}}^{\tau})^{\circ}\rtimes\sigma)_{ss}/\mbox{Inn}(\mathrm{\hat{\boldsymbol{Z}}^{\tau}(\hat{\boldsymbol{G}}^{\tau})^{\circ})}}.
\end{equation}
 Denote the $\sigma$-action on an element $g\in\hat{\boldsymbol{G}}$
by $g^{\sigma}$. Let $s,t\in\mathrm{\hat{\boldsymbol{T}}}^{\tau}$
be such that there exists $g\in\hat{\boldsymbol{Z}}^{\tau}(\hat{\boldsymbol{G}}^{\tau})^{\circ}$
satisfying $g^{-1}sg^{\sigma}=t$. As defined in Section \ref{new_sec},
let $\hat{\boldsymbol{B}}$ be the Borel subgroup of $\hat{\boldsymbol{G}}$
containing the maximal torus $\hat{\boldsymbol{T}}.$ Write $g=unv$
using Bruhat decomposion, where $u$, $v$ are in the unipotent radical
of $\hat{\boldsymbol{B}}$ and $n$ is in the normalizer $\hat{\boldsymbol{N}}$
of $\hat{\boldsymbol{T}}$. Also, $g=zg_{0}$ for some $z\in\hat{\boldsymbol{Z}}^{\tau}$
and $g_{0}\in\mathrm{(\hat{\boldsymbol{G}}^{\tau})^{\circ}}$. Let
$g_{0}=u_{0}n_{0}v_{0}$ be the Bruhat decomposition of $g_{0}$ in
$\mathrm{(\hat{\boldsymbol{G}}^{\tau})^{\circ}}$ with respect to
the Borel $(\hat{\boldsymbol{B}}^{\tau})^{\circ}$ and maximal torus
$(\hat{\boldsymbol{T}}^{\tau})^{\circ}$. Then $u=u_{0}$, $v=v_{0}$
and $n=zn_{0}$. Thus, 
\begin{eqnarray}
g^{-1}sg^{\sigma}=t & \implies & su^{\sigma}n^{\sigma}v^{\sigma}=unvt\nonumber \\
 & \implies & su^{\sigma}s^{-1}sn^{\sigma}v^{\sigma}=untt^{-1}vt\nonumber \\
 & \implies & sn^{\sigma}=nt\nonumber \\
 & \implies & sz^{\sigma}n_{0}^{\sigma}=zn_{0}t.\label{*}
\end{eqnarray}
Let $\bar{n}_{0}$ denote the image of $n_{0}$ in $\boldsymbol{W}$.
Then $n_{0}^{\tau}=n_{0}\implies\bar{n}_{0}\in\boldsymbol{W}^{\tau}$.
Also \ref{*} implies $\bar{n}_{0}^{\sigma}=\bar{n}_{0}$. Thus $\bar{n}_{0}\in\boldsymbol{W}^{\tau,\sigma}$.
Using again, the fact that $(\hat{\boldsymbol{G}}^{\tau})^{\circ}$
is reductive, it follows from the proof of Lemma 6.2 in \cite{MR546608}
that there exists $p\in N_{\mathrm{(\hat{\boldsymbol{G}}^{\tau})^{\circ}}}(\mathrm{\hat{\boldsymbol{T}}}^{\tau})^{\circ}$
such that $p^{\sigma}=p$ and $n_{0}\in p(\mathrm{\hat{\boldsymbol{T}}}^{\tau})^{\circ}$.
This fact is also shown in the proof of \cite[Lemma 4.7]{Zhu}. Let
$n_{0}=pq$ for some $q\in(\mathrm{\hat{\boldsymbol{T}}}^{\tau})^{\circ}$.
Then, 
\[
z^{-1}z^{\sigma}q^{-1}p^{-1}sp^{\sigma}q^{\sigma}=t.
\]
 Let $r=zq\in\mathrm{\hat{\boldsymbol{T}}}^{\tau}$. Then $r^{-1}p^{-1}spr^{\sigma}=t$.
Thus $\bar{s}=\bar{t}$ where $\bar{s}$ and $\bar{t}$ represent
the class of $s$ and $t$ in $\hat{\boldsymbol{T}}_{\sigma}^{\tau}/\boldsymbol{W}^{\tau,\sigma}=(\hat{\boldsymbol{T}}^{I})_{\sigma}/W$. 
\end{proof}
\end{prop}
Using Theorem \ref{referee results}$(4)$, we obtain,
\begin{cor}
$(\hat{\boldsymbol{T}}^{I})_{\sigma}/W\cong(\hat{\boldsymbol{G}}^{I}\rtimes\sigma)_{ss}/\mbox{Inn}(\hat{\boldsymbol{G}}^{I}).$
\end{cor}
Write $\boldsymbol{B}=\boldsymbol{T}\boldsymbol{U}$, where $\boldsymbol{B}$
is as before, the Borel subgroup of $\boldsymbol{G}$ containing $\boldsymbol{T},$
and where $\boldsymbol{U}$ is the unipotent radical of $\boldsymbol{B}.$
Using the Iwasawa decomposition \cite[Corr. 9.1.2]{haines2010satake},
we can define the spherical functions
\[
\Phi_{K,\chi}(tuk)=\chi(m)\delta^{1/2}(m),
\]
as in \cite{MR546593}, where $t\in T$, $u\in U$, $k\in K$ and
$\delta$ is the modulus function. The proof of the bijection 
\[
\{\pi|\pi\in\mbox{Rep}(G),\mbox{ \ensuremath{\pi^{K}\neq0\}}}\leftrightarrow\mbox{Hom}_{\mathbb{C}}(\mathscr{H}(\boldsymbol{G},K),\mathbb{C}).
\]
 is then identical to the case when $\boldsymbol{G}$ is unramified
and $K$ is hyperspecial, which is given in \cite{MR546593}. This
completes the proof of the Theorem \ref{main th}.

\section*{Acknowledgements}

I would like to thank my advisor Jiu-Kang Yu for his suggestion of
this problem and for his careful mentoring thoughout. I would also
like to thank Sungmun Cho for many helpful suggestions. I am thankful
to Sandeep Varma for pointing me to the reference for Theorem \ref{referee results}
and for his careful proof reading. 

I am very thankful to the referee for his many helpful inputs, particularly
in the proof of Proposition \ref{main_prop}. He pointed out the fact
that $\hat{\boldsymbol{Z}}^{I}(\hat{\boldsymbol{G}}^{I})^{\circ}$
is simply $ $$\hat{\boldsymbol{G}}^{I}$ and also provided a proof
of this fact. This makes the statement of Proposition \ref{main_prop}
appear more natural. 

\bibliographystyle{plain}
\bibliography{refrences1}

\end{document}